\documentclass[letterpaper, 10pt, conference]{ieeeconf}
\pdfoutput=1\IEEEoverridecommandlockouts 
\usepackage{amsmath,amssymb,url,color,multirow,float}
\usepackage{lipsum}
\usepackage{array} 
\usepackage{graphicx,subfigure,enumerate}

\newcommand{\m}{\mathrm{\mathbf{m}}} 
\newcommand{\al}{\alpha}
\newcommand{\be}{\beta}
\newcommand{\de}{\delta}

\newcommand{\ga}{\gamma}
\newcommand{\lam}{\lambda}
\newcommand{\td}{\tilde}

\newcommand{\diag}{\mathrm{diag}}

\newcommand{\I}{I_{3\times3}}
\renewcommand{\t}{\times}

\newcommand{\tr}{\mathrm{tr}}
\newcommand{\no}{\nonumber}
\newcommand{\SO}{\ensuremath{\mathsf{SO(3)}}}
\newcommand{\so}{\ensuremath{\mathfrak{so}(3)}}
\newcommand{\Sph}{\ensuremath{\mathsf{S}}}
\newcommand{\T}{\ensuremath{\mathsf{T}}}

\newcommand{\refeqn}[1]{(\ref{eq:#1})}
\renewcommand{\Re}{\ensuremath{\mathbb{R}}}
\newcommand{\w}{{\omega}}
\newcommand{\W}{{\Omega}}

\newcommand{\V}{\mathcal{V}}

\newcommand{\M}{\mathcal{M}}
\newcommand{\C}{\mathcal{C}}
\newcommand{\D}{\mathcal{D}}

\newcommand{\G}{\mathcal{G}}

\newcommand{\Ra}{\mathrm{I}}
\newcommand{\Rb}{\mathrm{II}}
\newcommand{\Rc}{\mathrm{III}}

\newcommand{\norm}[1]{\ensuremath{\left\| #1 \right\|}}

\DeclareMathOperator*{\argmin}{arg\,min}

\title{\LARGE \bf
Hybrid Attitude Observer on SO(3) with Global Asymptotic Stability}

\author{Tse-Huai Wu, Evan Kaufman and Taeyoung Lee\authorrefmark{1}%
\thanks{Tse-Huai Wu and Taeyoung Lee, Mechanical and Aerospace Engineering, The George Washington University, Washington DC 20052. {\tt \{wu52,evankaufman,tylee\}@gwu.edu}}%
\thanks{This research has been supported in part by NSF under the grant CMMI-1243000 (transferred from 1029551), CMMI-1335008, and CNS-1337722.}
}

\newtheorem{prop}{Proposition}

\graphicspath{{./Figures/}}
\begin{document}
\allowdisplaybreaks
\maketitle \thispagestyle{empty} \pagestyle{empty}

\begin{abstract}
This paper presents a deterministic hybrid observer for the attitude dynamics of a rigid body that guarantees global asymptotical stability. Any smooth attitude observer suffers from the inherent topological restriction that it is impossible to achieve global attractivity, and as such, attitude observers have been developed with almost global asymptotic stability. We demonstrate that such attitude observer may yields very slow initial convergence rates, and motivated by this, we propose a hybrid attitude observer that guarantees global asymptotic stability on the special orthogonal group. We illustrate that the proposed observer exhibits substantially improved convergence rate uniformly via numerical examples and experiments. 
%
\end{abstract}

\section{Introduction}
Attitude estimation has been intensively studied with various filtering approaches and assumptions~\cite{Crassidis07}. They can be categorized with the choice of attitude representations. It is well known that minimal, three-parameter attitude representations, such as Euler-angles or modified Rodriguez parameters, suffer from singularities. They are not suitable for large angle rotational maneuvers, as the type of parameters should be switched persistently in the vicinity of singularities. 

Quaternions are another popular choice in attitude estimation~\cite{Bar1985,Markley94,Psiaki2000}. They do not have singularities, but as the configuration space of quaternions, namely the three-sphere double covers the special orthogonal group of the attitude configuration space, there exists ambiguity. This implies that a single attitude may be represented by two antipodal points on the three-sphere. The ambiguity should be carefully resolved in any quaternion-based attitude observer and controller, otherwise they may exhibit unwinding, where an initial attitude estimate goes through unnecessarily large rotations even if the initial estimate error is small, or it becomes sensitive to noise~\cite{BhaBerSCL00}. To resolve this consistently, an additional mechanism to lift the measurement of attitudes into the three-sphere has been introduced~\cite{MaySanITAC11}.


Instead, attitude observers have been designed directly on the special orthogonal group to avoid both singularities of local coordinates and the ambiguity of quaternions. These include, for example, complementary filters~\cite{Hamel2006,Mahony08}, an observer in the presence of angular measurement noise~\cite{Vas2009}, and attitude estimation with single vector measurements~\cite{Alireza10,Batista14B}. However, these results are commonly restricted by the topological restriction of the special orthogonal group, prohibiting global attractivity for smooth attitude flows~\cite{BhaBerSCL00}. As such, they guarantee \textit{almost} global asymptotic stability instead, where the region of attraction excludes a certain set of zero measure, which corresponds to the stable manifold to undesired equilibria. This is considered as the strongest stability property that can be achieved by smooth attitude estimators. 

The fact that the region of attraction does not cover the entire configuration manifold does not seem to be a major issue in practice, as the probability that an initial condition exactly lies in such undesired set is zero, provided that the initial condition is chosen randomly. But, the existence of those undesired equilibria may have nontrivial effects on the attitude dynamics~\cite{LeeLeoPICDC11}. In particular, attitude estimators may yield significantly poor convergence rates, especially with large initial attitude estimator errors.  As a result, an initial attitude estimate that is close to the exact opposite of the true attitude may not be corrected for a while, thereby causing significant performance degradations.

Recently, in attitude control systems, the topological restriction of the special orthogonal group has been addressed via hybrid control system approaches~\cite{MaySanITAC11,Mayhew13}, where a hysteresis-based switching algorithm is introduced to achieve global asymptotic stability. A new set of attitude error functions is introduced in~\cite{LeeGlobal,LeeITAC15} to guarantee stronger global exponential stability with an explicit and compact form of stability criteria. 

The main objective of this paper is to achieve global asymptotic stability in attitude estimation, while addressing the ambiguity of quaternions and the topological obstruction of the special orthogonal group concurrently. The synergistic attitude error functions introduced in~\cite{LeeGlobal,LeeITAC15} are based on two preselected reference directions that should be orthogonal with each other, and they are not suitable for attitude estimation that are based on attitude measurements or line-of-sight measurements toward, possibly many, arbitrary reference objects. We first introduce revised synergistic attitude error functions without such restrictions of~\cite{LeeGlobal,LeeITAC15}, and apply them to attitude estimation while compensating the effects of a fixed bias in angular velocity measurements. The key idea is designing a set of attitude error functions such that attitude estimates are expelled from undesired equilibria to achieve global asymptotic stability. 

Perhaps, the only other attitude observers that guarantee global attractivity corresponds to the work in~\cite{Batista12,Batista14A}. However, it is considered that the special orthogonal group is embedded in a linear space, and an observer is designed on the linear space to overcome the topological restriction. Therefore, the estimated attitude does not necessarily evolve on the special orthogonal group. In short, attitude observers constructed on the special orthogonal group that guarantee global asymptotic stability have been unprecedented, and it is the unique contribution of this paper. The desirable properties of the proposed approach are illustrated by both numerical examples and experimental results.

Another contribution of this paper is developing a numerical algorithm to implement the proposed attitude observer while preserving the structures of the special orthogonal group. It is well known that conventional numerical integrators, such as the Runge-Kutta method and its variations, do not conserve the orthogonality of rotation matrices, thereby yielding attitude estimates that are incorrect geometrically~\cite{Hairer2000}. This paper construct a numerical algorithm 
based on the Lie group method~\cite{IseMunAN00}, to guarantee that the estimated attitudes evolve on the special orthogonal group. The desirable properties of the proposed approach are illustrated by both numerical examples and experimental results. 

The preliminary results of this paper have been presented in~\cite{Wu2015}. This paper include experimental verification, geometric numerical algorithm, and stability proof that have not discussed in~\cite{Wu2015}. The paper is organized as follows: an attitude estimation problem is formulated, and prior results are summarized at Sections II, and III respectively. A hybrid attitude observer guaranteeing global asymptotic stability is proposed at Section IV, followed by numerical examples with geometric numerical integration techniques and experimental verification.



\section{Problem Formulation}\label{sec:Dynamics}

Consider the attitude dynamics of a rigid body. Two coordinate frames are defined: an inertial reference frame and a body-fixed frame. The attitude of the rigid body is described by a rotation matrix $R\in\SO$ that represents the transformation of a representation of a vector from the body-fixed frame to the inertial reference frame. The configuration manifold of attitude is the special orthogonal group:
	\begin{align*} 
	\SO=\{R\in\Re^{3\t3}\,|\,{R^\T}R=I,\,\det[R]=1\}.
	\end{align*}
Let $\w\in\Re^3$ and $\W\in\Re^3$ denote the angular velocities of the rigid body, represented with respect to the inertial reference frame, and the body-fixed frame, respectively.	
The attitude kinematics equations are given by
	\begin{gather} 
	\dot{R}=\hat{\w}R=R\hat{\W}, \label{eq:eom2}
	\end{gather} 
where the \textit{hat} map $\wedge:\Re^3\rightarrow\so$ transforms a vector in $\Re^3$ to a $3\t3$ skew-symmetric matrix such that $\hat{x}y=x^\wedge y=x\t y$ for any $x,\,y\in\Re^3$. And the inverse of hat map is denoted by the \textit{vee} map $\vee:\so\rightarrow\Re^3$. Several properties of hat map are listed as follows:
	\begin{gather} 
	\widehat{x\t y}=\hat{x}\hat{y}-\hat{y}\hat{x}=yx^\T-xy^\T, \label{eq:widehat}\\
	\tr[A\hat{x}]=\tr[\hat{x}A]=-x^\T(A-A^\T), \label{eq:hat1}\\
	R\hat{x}R^\T =(Rx)^\wedge, \label{eq:hat2}\\
	\hat{x}A+A^\T\hat{x}=(\{\tr[A]\I-A\}x)^\wedge \label{eq:hat3},
	\end{gather}
for any $x\in\Re^3,\,A\in\Re^{3\t3},\,R\in\SO$.  The standard inner product of two vectors is denoted by $x\cdot y=x^\T y$. Throughout this paper, $\I$ denotes the $3\t3$ identity matrix and the 2-norm of matrix $A$ is denoted by $\|A\|$.

The measurement model is as follows. Let $v^I_i\in\Sph^2=\{v\in\Re^3\,|\, \norm{v}=1\}$ be the unit-vector from the origin of the body-fixed frame along the $i$-th reference direction represented with respect to the inertial frame. They may be the direction of the gravity, the direction of the magnetic field, or the line-of-sight toward a distinctive point. It is assumed that there are $n>2$ distinct, known reference directions. Also, they are normalized such that $\|v^I_i\|=1$. The rigid body is equipped with sensors to measure such reference directions, such as an accelerometer or a magnetometer, and the corresponding $i$-th measurement, namely $v_i^B\in\Sph^2$ is obtained with respect to the body-fixed frame as
	\begin{align} 
	v_i^B=R^\T v_i^I. \label{eq:vb}
	\end{align}

Also, it is assumed that the angular velocity is measured with a gyro with a unknown but fixed bias, such that the measured angular velocity $\Omega_y\in\Re^3$ is written as
	\begin{gather} 
	\W_y=\W+\ga, \label{eq:Wy}
	\end{gather} 
where $\ga\in\Re^3$ denotes the constant gyro bias. In short, the measurements correspond to $\{v_1^B,\ldots v_n^B\}$  with $n>2$, and $\Omega_y$. We wish to design an observer to estimate the attitude of the rigid body $R$, and the gyro bias $\ga$, based on these measurements.

\section{Almost Global Attitude Observer on $\SO$} \label{sec:obs}

In this section, we first review the the passive complementary filter with bias correction presented in~\cite{Mahony08}. This attitude estimator is shown to guarantee almost global asymptotic stability, where there exist three additional undesired equilibria along the flows of the filter. 

\subsection{Estimate Frame and Error Variables}

Define an orthonormal frame estimated by the observer. The orientation and angular velocity of the estimate frame with respect to the inertial reference frame are denoted by $\bar R\in\SO$ and $\bar\w\in\Re^3$, respectively. More explicitly, $\bar{R}$ denotes the linear transformation from the inertial reference frame to the estimated frame. 

The discrepancy between the true attitude $R$ and the estimated attitude $\bar{R}$ is denoted by a rotation matrix $\td{R}\in\SO$ defined as
	\begin{align} 
	\td{R}=\bar{R}^\T R,
	\end{align}
which corresponds to the linear transformation from the body-fixed frame to the estimated frame, and  $\td{R}=\I$ when $\bar{R}=R$. 

The vector $v_i^I$ representing the known, reference direction can be expressed with respect to the estimated frame to obtain,
	\begin{align} 
	v_i^E=\bar{R}^\T v_i^I. \label{eq:ve}
	\end{align}	
Note that $v_i^E=v_i^B$ when there is no estimation error, i.e., $\tilde{R}=I_{3\times 3}$. Define an error function representing the discrepancy between $v_i^E$ and $v_i^B$ as
	\begin{align} 
	\Psi_i=1-{v_i^E}^\T v_i^B, \label{eq:psiI}	
	\end{align}
which can be rewritten from \refeqn{vb} and \refeqn{ve} as
	\begin{align*} 
	\Psi_i=1-{v_i^E}^\T v_i^B =1-\tr[v_i^E{v_i^B}^\T]= 1-\tr[\bar{R}^\T v_i^I{v_i^I}^\T R].
	\end{align*}
For distinct positive constants $k_1,\ldots,k_n$, the overall configuration error function $\Psi\in\Re$ is defined as the weighted sum,
	\begin{align} 
	\Psi=\sum_{i=1}^n k_i\Psi_i =\sum_{i=1}^n k_i-\tr\Big[\bar{R}^\T K R\Big], \label{eq:Psi}
	\end{align}
where the symmetric matrix $K\in\Re^{3\times 3}$ is given by
	\begin{align} 
	K=\sum_{i=1}^n k_iv_i^I{v_i^I}^\T=K^\T.\label{eq:K}
	\end{align}
As it is symmetric, it can be decomposed into
	\begin{align} 
	K=UGU^\T, \label{eq:KUGU}
	\end{align}
where the orthogonal matrix $U\in\Re^{3\times 3}$ are composed of the normalized eigenvectors of $K$, and the diagonal matrix $G=\diag(\lam_1,\lam_2,\lam_3)$ is defined by the real non-negative eigenvalues $\lambda_1,\lambda_2,\lambda_3\in\Re$ of $K$. When $n>3$, the matrix $K$ is positive-definite, and all of the eigenvalues are strictly positive. When $n=2$, there is a single zero eigenvalue.  Without loss of generality, we can assume $U\in\SO$ by reordering. For example, if $\mathrm{det}(U)=-1$, the first two diagonal elements of $G$, and the first two columns of $U$ can be swapped such that $\mathrm{det}(U)=+1$. 

%

In addition, let the estimated gyro bias $\bar{\ga}\in\Re^3$ and the bias estimation error be defined as 
	\begin{align} 
	\td{\ga}=\ga-\bar{\ga}.
	\end{align}

\subsection{Complementary Filter}
The complementary filter~\cite{Mahony08} is defined as
	\begin{gather} 
	\dot{\bar R}=\bar{R}[(\W_y-\bar{\ga})+k_Re_R]^\wedge, \label{eq:ob1}\\ 
	\dot{\bar \ga} =-k_Ie_R,  \label{eq:ob2}\\
	e_R=\sum_{i=1}^nk_iv_i^B\t v_i^E,\label{eq:ob3}
	\end{gather}
where $k_R,k_I\in\Re$ are positive constants and $e_R\in\Re^3$ is the innovation term.

\begin{prop}{\cite{Mahony08}}\;
Consider the attitude kinematics \refeqn{eom2} with the measurements \refeqn{vb} and \refeqn{Wy}. The attitude observer given by \refeqn{ob1}, \refeqn{ob2} and \refeqn{ob3} satisfies the following properties.
\begin{itemize}
\item There are four equilibria, given by
	\begin{align} 
	(\bar{R},\,\bar\ga)\in\{(R,\,\ga),\,(UD_iU^\T R,\,\ga)\}.\label{eq:UD}
	\end{align}					
for  $D_1=\diag[1,-1,-1]$, $D_2=\diag[-1,1,-1]$ and $D_3=\diag[-1,-1,1]$.
\item The desired equilibrium $(\bar{R},\bar{\ga})=(R,\ga)$ is almost globally asymptotically stable. 
\end{itemize}
\end{prop}

In summary, this observer guarantees that the estimated attitude asymptotically converge to the true attitude for almost all cases, excluding the estimated attitudes starting from a set of zero measure that corresponds to the union of the stable manifolds to the above three undesired equilibria. While there is no chance in practice that the observer is initialized by such a \textit{thin} set which does not yield asymptotic convergence, the existence of such stable manifold to the undesired equilibria may have strong effects on the estimator dynamics~\cite{LeeLeoPICDC11}. More explicitly, the convergence rate may become very slow near the undesired equilibria, and it is particularly undesirable as undesired equilibria represent large estimation errors, i.e., the estimated attitude is opposite to the true attitude. This motivates the subsequent development for a hybrid attitude estimator with global asymptotic stability. 


\section{Global Attitude Observer on $\SO$} \label{sec:global}

Hysteresis-based switching algorithms have been introduced to achieve global asymptotic stability~\cite{MaySanITAC11,Mayhew13}, and global exponential stability~\cite{LeeGlobal,LeeITAC15} in attitude controls. In particular, the set of attitude error functions formulated in~\cite{LeeGlobal,LeeITAC15} has desirable properties that guarantee stronger global exponential stability with an explicit and compact form of stability criteria. But, it is constructed with two pre-defined reference directions that are assumed to be orthogonal with each other, and therefore, it is not suitable for the presented attitude estimation problem that is based on an arbitrary number of unit-vector measurements.  

In this section, we present a new set of synergistic attitude error functions, and utilize it for the attitude estimation problem with a gyro-bias correction. The key idea is that the error functions are designed such that the attitude estimate is expelled from the vicinity of the undesired equilibria. 


%

\subsection{Estimate Frame and Error Variables}

We first derive additional properties of the attitude error function and error variables defined at the previous section. Let $u_i\in\Sph^2$ be the $i$-th column of the matrix $U$, i.e., $U=[u_1,u_2,u_3]\in\SO$, and therefore, $u_3=u_1\times u_2$. Also, let $b_i,\bar{b}_i\in\Sph^2$ be the representations of $u_i$ with respect to the body-fixed frame, and the estimated frame, respectively, i.e., 
	\begin{align} 
	b_i=R^\T u_i \quad\text{and}\quad \bar{b}_i=\bar {R}^\T u_i. \label{eq:sRb}
	\end{align}			
Note that $b_3=b_1\times b_2$, and $\bar b_3=\bar b_1\times \bar b_2$ as well. 

The unit-vectors $b_i$ can be obtained directly from the measurements $v_i^B$ as follows. Define a symmetric matrix $K_B\in\Re^{3\times 3}$ as
\begin{align}
K_B=\sum_{i=1}^n k_i v_i^B {v_i^B}^\T = R^T K R, \label{eq:KB}
\end{align}
where \refeqn{vb} and \refeqn{K} are used for the second equality. Substituting \refeqn{KUGU}, \refeqn{sRb}, this can be written as
\begin{align}
K_B=R^T UGU^\T R \triangleq B G B^\T, \label{eq:B}
\end{align}
where the matrix $B=R^TU$ is composed of $b_i$ from \refeqn{sRb}, i.e., $B=[b_1,b_2,b_3]\in\SO$. In short, for a given set of measurements, $v_i^B$, defining the matrix $K_B$ as \refeqn{KB}, and decomposing it as \refeqn{B} yields the orthonormal unit-vectors $b_1,b_2,b_3$. Their estimates $\bar b_1,\bar b_2,\bar b_3$ are obtained by the second equation of \refeqn{sRb}. 

From \refeqn{K}, we have $\tr[K]=\sum_{i=1}^n\tr[k_i v^I_i{v^I}^\T_i]=\sum_{i=1}^n k_i$, and on the other hand, from \refeqn{KUGU}, $\tr[K]=\tr[G]=\sum_{i=1}^3\lambda_i$. Also, \refeqn{KUGU} is rewritten into $K=\sum_{i=1}^n \lambda_i u_i u_i^\T$. Using these, the attitude error function $\Psi$ defined at \refeqn{Psi} can be rearranged as
\begin{align}
\Psi&= \sum_{i=1}^3 \lambda_i - \tr[\bar R^\T \lambda_i s_i s_i^\T R] 
 =\sum_{i=1}^{3}\lam_i(1-\bar{b}_i^\T b_i). \label{eq:npsi}
\end{align}
Therefore, the attitude error function is obtained by comparing $b_i$ with $\bar b_i$.

Next, we show that the undesired equilibria defined at \refeqn{UD} can be expressed in terms of $b_i$. Since $s_i$ are column vectors of $U$, it can be shown that $UD_iU^\T=\exp(\pi\hat{u}_i)$. Therefore, the undesired equilibria can be written as
\begin{align*}
UD_iU^\T R =\exp(\pi\hat{u}_i)R=R\exp(\pi\hat{b}_i),
\end{align*}
where we have used the fact that $\exp(\widehat{Rx})=R\exp(\hat{x})R$, for any rotation matrix $R\in\SO$ and any vector $x\in\Re$. This shows that the $i$-th undesired equilibrium is obtained by rotating the true attitude about $b_i$ by $180^\circ$. The values of $b_i$ at undesired equilibria are summarized later at Table I.

\newcommand{\mb}{\m}

\subsection{Expelling Attitude Error Functions}

Next, we design a hybrid attitude observer to avoid the undesired equilibria to achieve global attractivity. We first introduce a mathematical formulation of hybrid systems as follows~\cite{GoeSanICSM09}. Let $\mathcal{M}$ be the set of discrete modes, and let $\mathcal{Q}$ be the domain of continuous states. Given a state $(\mb,\xi)\in\mathcal{M}\times \mathcal{Q}$, a hybrid system is defined by
\begin{alignat}{2}
\dot \xi & = \mathcal{F}(\mb,\xi),&\quad (\mb,\xi)&\in\mathcal{C},\label{eqn:Hyb1}\\
\mb^+ & = \mathcal{G}(\mb,\xi),& (\mb,\xi)&\in\mathcal{D},\label{eqn:Hyb2}
\end{alignat}
where the flow map $\mathcal{F}:\mathcal{M}\times\mathcal{Q}\rightarrow\mathcal{Q}$ describes the evolution of the continuous state $\xi$; the flow set $\mathcal{C}\subset\mathcal{M}\times\mathcal{Q}$ defines where the continuous state evolves; the jump map $\mathcal{G}:\mathcal{M}\times\mathcal{Q}\rightarrow\mathcal{M}$ governs the discrete dynamics; the jump set $\mathcal{D}\subset\mathcal{M}\times\mathcal{Q}$ defines where discrete jumps are permitted.


The proposed hybrid attitude observer is composed of one nominal mode and two distinct modes. Define the $i$-th nominal error function as
\begin{align}
\Psi_{N_i} = 1-\bar b_i^\T b_i,\label{eq:Ni}
\end{align}
for $i=\{1,2,3\}$, where its definition is motivated by \refeqn{npsi}. In the vicinity of the undesired equilibrium where $\bar b_1=-b_1$, the error function is switched into the following expelling error function,
\begin{align}
\Psi_{E_1} = \al+\be\bar{b}_1^\T(b_1\t b_2) =\al+\be\bar{b}_1^\T b_3,\label{eq:E1}
\end{align}
for constants $\alpha,\beta$ satisfying $1<\al<2$ and $\|\be\|<\al-1$. The attitude observer designed with the above expelling error function steers the estimated direction $\bar b_1$ toward a direction normal to $-b_1$, namely $-\frac{\beta}{|\beta|}b_1\times b_2$, such that the estimated attitude is rotated away from the undesired equilibrium. Similarly, the following expelling error function is engaged near the undesired equilibria where $\bar b_2=-b_2$,
\begin{align}
	\Psi_{E_2}=\al+\be\bar{b}_2^\T(b_1\t b_2) =\al+\be\bar{b}_2^\T b_3.\label{eq:E2}
	\end{align}

More explicitly, there are three modes $\mathcal{M}=\{\Ra,\Rb,\Rb\}$ in the proposed hybrid attitude observer, and the attitude error function for each mode is given by
	\begin{gather} 
	\Psi_\Ra(\bar R)=\lam_1\Psi_{N_1} +\lam_2\Psi_{N_2} +\lam_3\Psi_{N_3}, \label{eq:I}\\
	\Psi_\Rb(\bar R)=\lam_1\Psi_{N_1} +\lam_2\Psi_{E_2} +\lam_3\Psi_{N_3}, \label{eq:II} \\
	\Psi_\Rc(\bar R)=\lam_1\Psi_{E_1} +\lam_2\Psi_{N_2} +\lam_3\Psi_{N_3}. \label{eq:III}
	\end{gather}


Next, to formulate the switching logic of the proposed hybrid system, a variable $\rho$ is defined as the minimum attitude error among three modes,
	\begin{align} 
	\rho(\bar R)=\min_{\m\in\M}\{\Psi_\m(\bar R)\},\label{eq:rho}
	\end{align}
for $\M\in\{\Ra,\Rb,\Rc\}$. The jump map is chosen such that the discrete mode is switched into the new mode that yields the above minimum error, i.e., 
	\begin{align} 
	\G(\bar R)=\argmin_{\m\in\M}{\Psi_\m(\bar R)}=\{\m\in\M\,:\,\Psi_\m=\rho\}.\label{eq:GG}
	\end{align}
It is possible to switch whenever $\Psi_\m > \rho$. However, it may cause chattering due to measurement noise. Instead, we introduce a positive hysteresis gap $\de\in\Re$ for improved robustness, and the switching is triggered if the difference between the current configuration error and the minimum value is greater than $\de$. This leads to the following formulation of the jump set and the flow set:
	\begin{gather} 
	\D=\{(\bar{R},\m):\Psi_\m-\rho\geq\de\}, \label{eq:DD}\\
	\C=\{(\bar{R},\m):\Psi_\m-\rho<\de\}.\label{eq:CC}
	\end{gather}

\subsection{Hybrid Attitude Observer}
For positive constants $k_R,k_I$, the proposed hybrid attitude observer is defined as
	\begin{gather} 
	\dot{\bar R}=\bar{R}[(\W_y-\bar{\ga})+k_Re_H]^\wedge, \label{eq:gob1}\\ 
	\dot{\bar \ga} =-k_Ie_H,  \label{eq:gob2}\\
	e_H=\sum_{i=1}^3\lam_ie_{H_i},\label{eq:gob3}
	\end{gather}
where the $i$-th hybrid innovation terms $e_{H_i}$ are given by
\begin{align}
  e_{H_1}&=\begin{cases}
               b_1\t\bar{b}_1 & \text{if}\quad \m=\Ra,\Rb,\\
               -\be(b_3\t\bar{b}_1) & \text{if}\quad \m=\Rc,\\
            \end{cases} 
  \\
  e_{H_2}&=\begin{cases}
               b_2\t\bar{b}_2 & \text{if}\quad \m=\Ra,\Rc,\\
               -\be(b_3\t\bar{b}_2) & \text{if}\quad \m=\Rb,\\
            \end{cases}
  \\          
  e_{H_3}&= b_3\t\bar{b}_3  \qquad\qquad\text{for}\quad \m=\Ra,\Rb,\Rc, \label{eq:eH3}
\end{align}
which are obtained by taking the derivatives of the attitude error functions defined at \refeqn{I}-\refeqn{III}. Note that the observer in the nominal mode $\m=\Ra$ is the same as \refeqn{ob1}-\refeqn{ob3}. 

\begin{prop} \label{prop:delta}
Consider the attitude kinematics \refeqn{eom2} with the measurements \refeqn{vb} and \refeqn{Wy}. The hybrid attitude observer is defined by \refeqn{rho}-\refeqn{eH3}, where $1<\al<2$ and $|\be|<\al-1$. Choose the hysteresis gap $\de$ such that 
	\begin{align} 
	0<\de<\min\{\lam_1,\,\lam_2\}\min\{2-\al,\,\al-|\be|-1\}, \label{eq:de}
	\end{align}
where $\lambda_i$ is defined at \refeqn{KUGU}. Then, the desired equilibrium $(\bar{R},\bar{\ga})=(R,\ga)$ is globally asymptotically stable, and the number of discrete jumps is finite.
\end{prop}

\begin{proof}
See Appendix A.
\end{proof}

The proposed hybrid attitude observer guarantees that the estimated attitude and the estimated gyro bias asymptotically converge to their true values globally. Compared with~\cite{Batista12,Batista14A}, these are constructed directly on $\SO$ such that the estimated attitude $\bar R$ lies in $\SO$ always. Furthermore, the proposed hybrid attitude observer exhibits substantially better convergence rate than other attitude observers guaranteeing almost global asymptotic stability, such as~\cite{Mahony08}, at certain cases as illustrated below.


\section{Numerical Analysis}

General purpose numerical integration techniques, such as the Runge-Kutta method, are commonly applied to implement any estimation algorithm numerically. However, conventional numerical integrators are not suitable for the proposed attitude observer, as the attitude estimate computed such numerical integrators are not orthogonal in general~\cite{Hairer2000}, thereby yielding geometrically incorrect attitude estimates. Here, we present a numerical algorithm for the proposed attitude observer, to ensure that the corresponding attitude estimates evolve on the special orthogonal group numerically.  


\subsection{Geometric Numerical Integration}

The proposed algorithm is based on the Lie group~\cite{IseMunAN00}, where the attitude estimates are updated by the group operation of $\SO$, namely the matrix multiplication to ensure that they evolve on $\SO$ numerically.

First, define the estimate angular velocity $\bar{\W}\triangleq\W_{y}-\bar{\ga}+k_Re_{H}\in\Re^3$. Let $\bar{R}_{n}$ and $\bar{\W}_n$ denote the estimate attitude and angular velocity at the $n$-th time step, respectively. The intermediate updates are given by
	\begin{gather*} 
	\bar{R}_{n+1}'=\exp^{hK_{{\bar R}_1}}\bar{R}_{n},\quad \bar\ga_{n+1}'=\bar\ga_n+hK_{1_\ga},
	\end{gather*}
where $h\in\Re$ is the step size, and the intermediate increments $K_{{\bar R}_1}\in\Re^{3\t3}$ and $K_{{\bar\ga}_1}\in\Re^3$ are defined as
	\begin{gather*} 
	K_{{\bar R}_1} =(\bar{R}_{n}\bar{\W}_n)^\wedge,\quad K_{{\bar\ga}_1}=-k_Ie_{H_n}(R_n,\bar{R}_n),
	\end{gather*}
In contrast to the conventional Runge-Kutta method, that would yield $\bar{R}_{n+1}'=\bar{R}_{n}+hK_{{\bar R}_1}$, here $\bar{R}_{n+1}'$ is constructed by the product of two rotation matrices to ensure $\bar{R}_{n+1}'$ is orthogonal numerically.



Next, the second stage increments are evaluated by
	\begin{gather*} 
	K_{{\bar R}_2}=(\bar{R}_{n+1}'\bar{\W}_{n+1}')^\wedge, \quad K_{\bar\ga_2}=-k_Ie_{H_{n+1}}(R_{n+1},\bar{R}_{n+1}'),
	\end{gather*} 
where $\bar{\W}_{n+1}'=\W_{y_{n+1}}-\bar{\ga}_{n+1}'+k_Re_{H_{n+1}}\in\Re^3$. Combining the increments from both of the stages, we can obtain $\bar{R}$ and $\bar\ga$ for the $(n+1)$-th step:
	\begin{gather*} 
	\bar{R}_{n+1}=\exp^{\frac{1}{2}h(K_{{\bar R}_1}+K_{{\bar R}_2})}\bar{R}_{n}, \\ \bar{\ga}_{n+1}=\bar\ga_n+\frac{1}{2}h(K_{\bar{\ga}_1}+K_{\bar{\ga}_2}).
	\end{gather*} 
These correspond to the second-order Crouch-Crossman method~\cite{Hairer2000}.


\subsection{Numerical Examples} \label{subsec:num}
We select three linearly independent vectors in the inertial frame, $v_i^I=v_i'^I/\|v_i'^I\|$, for $i={1,2,3}$, where $v_1'^I=[-2~5~2]^\T$  $v_2'^I=[10~-1~0]^\T$ and $v_3'^I=[0~1~-2]^\T$. The true attitude trajectory is selected in terms of 3-2-1 Euler angles as
	\begin{align*} 
	R(\mathfrak{a}(t),\mathfrak{b}(t),\mathfrak{c}(t)) =\exp(\mathfrak{a}\hat{e}_3)\exp(\mathfrak{b}\hat{e}_2)\exp(\mathfrak{c}\hat{e}_1),
	\end{align*}
where $\mathfrak{a}=\sin(0.5t)$, $\mathfrak{b}=2\sin(t)$, $\mathfrak{c}=\cos(2t)-3$. These represent a non-trivial complex rotational maneuver.  The initial value of estimate attitude and estimate bias are given by 
	\begin{align*} 
	\bar{R}(0)=\left[\begin{array}{ccc}0.2527&-0.8907&-0.3779\\ 0.6381&0.4470&-0.6270\\ 0.7273&-0.0827&-0.6813\end{array}\right],
	\end{align*}	
and $\bar\ga=[0.0997~ -0.1042~ 0.2027]^\T$, respectively. The controller parameters are chosen as $k_1=1.211$, $k_2=1.21$, $k_3=1.209$, $k_R=1 $, $k_I=0.25$, $\al=1.9$ and $\be=0.899$.

The simulation results for the smooth complementary attitude filter~\cite{Mahony08}, and the proposed hybrid attitude observer are illustrated at Figure \ref{fig:ob-2} for the first case when there is no gyro bias, i.e., $\gamma=0$, and at Figure \ref{fig:ob-1} for the second case with a non-zero gyro bias, $\ga=[0.1~-0.1~0.2]^\T$, respectively. They are computed via the proposed numerical integration algorithm with the time step $h=0.05~\mathrm{sec}$. For both cases, the attitude estimation error of the complementary filter does not change for the first few seconds even though the estimate error is very large, but the proposed hybrid attitude observer exhibits a substantially faster convergence rate by engaging the third, switching mode initially and switching back to the normal mode later at $t=1.40$ seconds and $t=1.15$ seconds, respectively. 

We further compare the proposed geometric numerical integration algorithm to a conventional 4/5th Runge-Kutta method for the second case. Simulation results at Figure~\ref{fig:compare} show that two methods perform well over the first 20 seconds. However, the attitude estimation error of the  Runge-Kutta method increases noticeably later. This is caused by the fact that the computed $\bar R$ deviates from $\SO$, and such error may accumulate as illustrated in Figure~\ref{fig:compare}(b). 

\begin{figure}[h]
\centerline{
	\subfigure[Attitude estimation error $\|\bar{R}-R\|$]{\includegraphics[width=0.53\columnwidth]{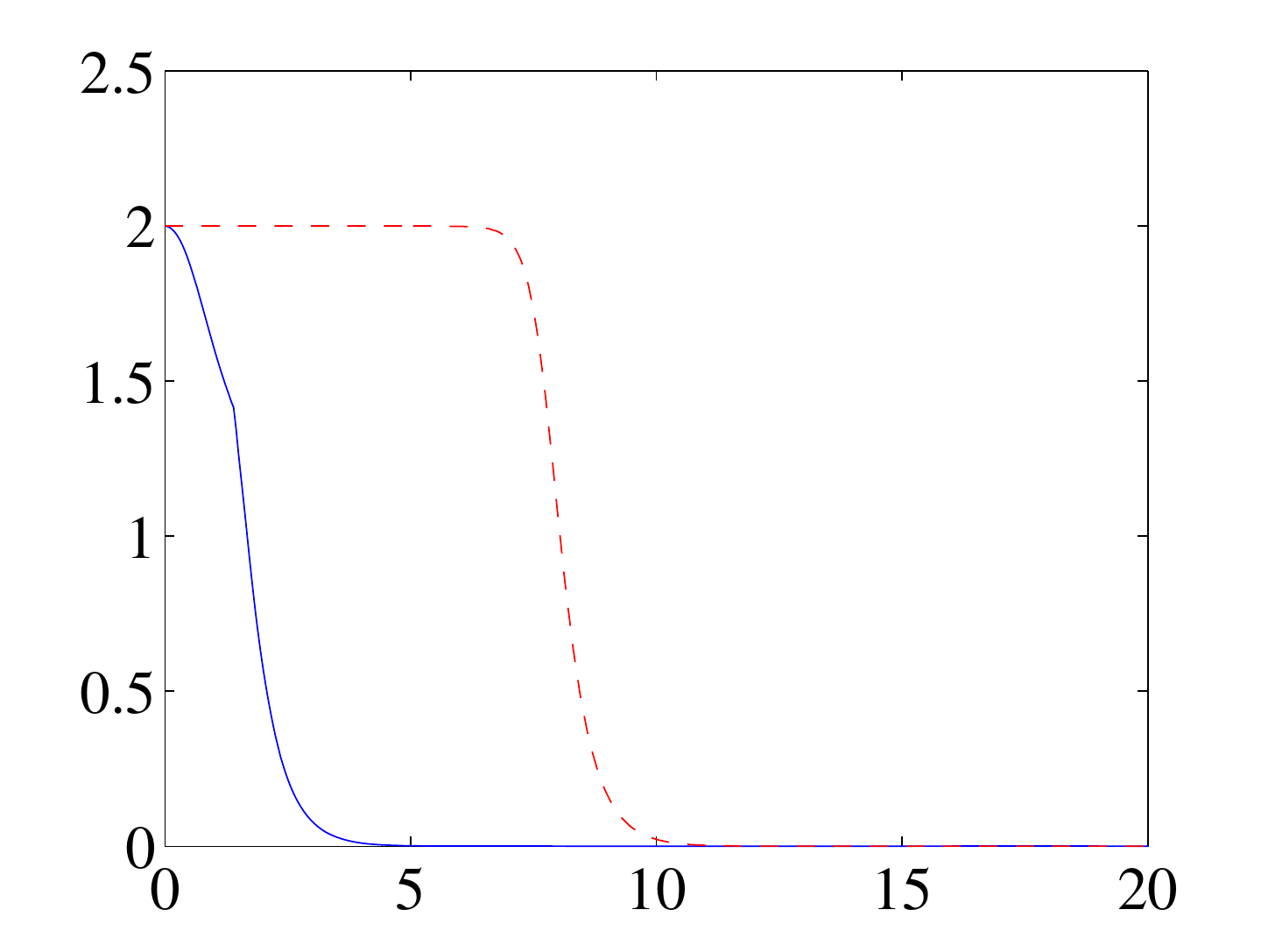}}
	\hspace*{-0.015\columnwidth}
	\subfigure[Mode change]{\includegraphics[width=0.53\columnwidth]{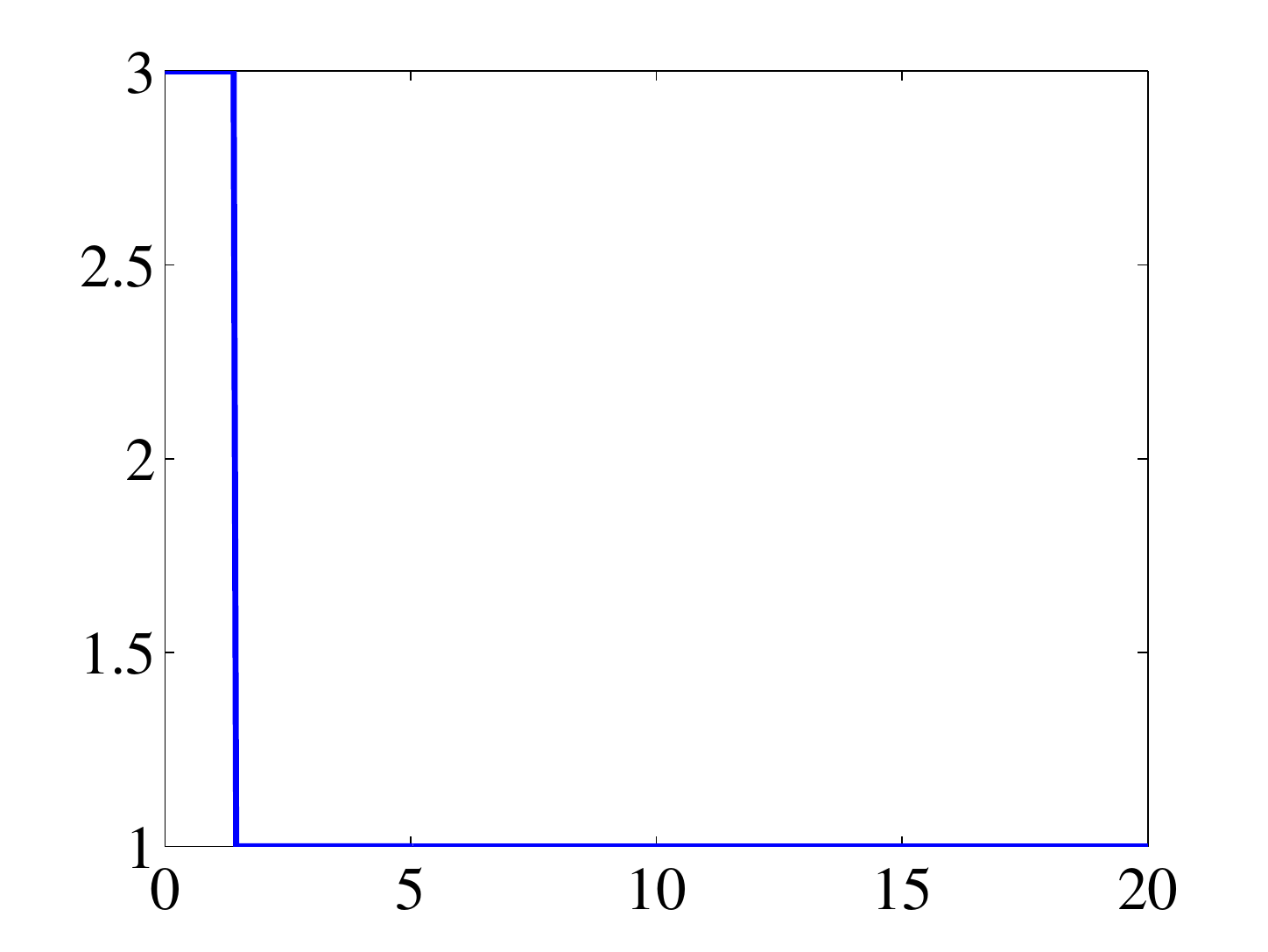}}
}	
\centerline{
	\subfigure[Innovation term $e_H(e_R)$]{\includegraphics[width=0.53\columnwidth]{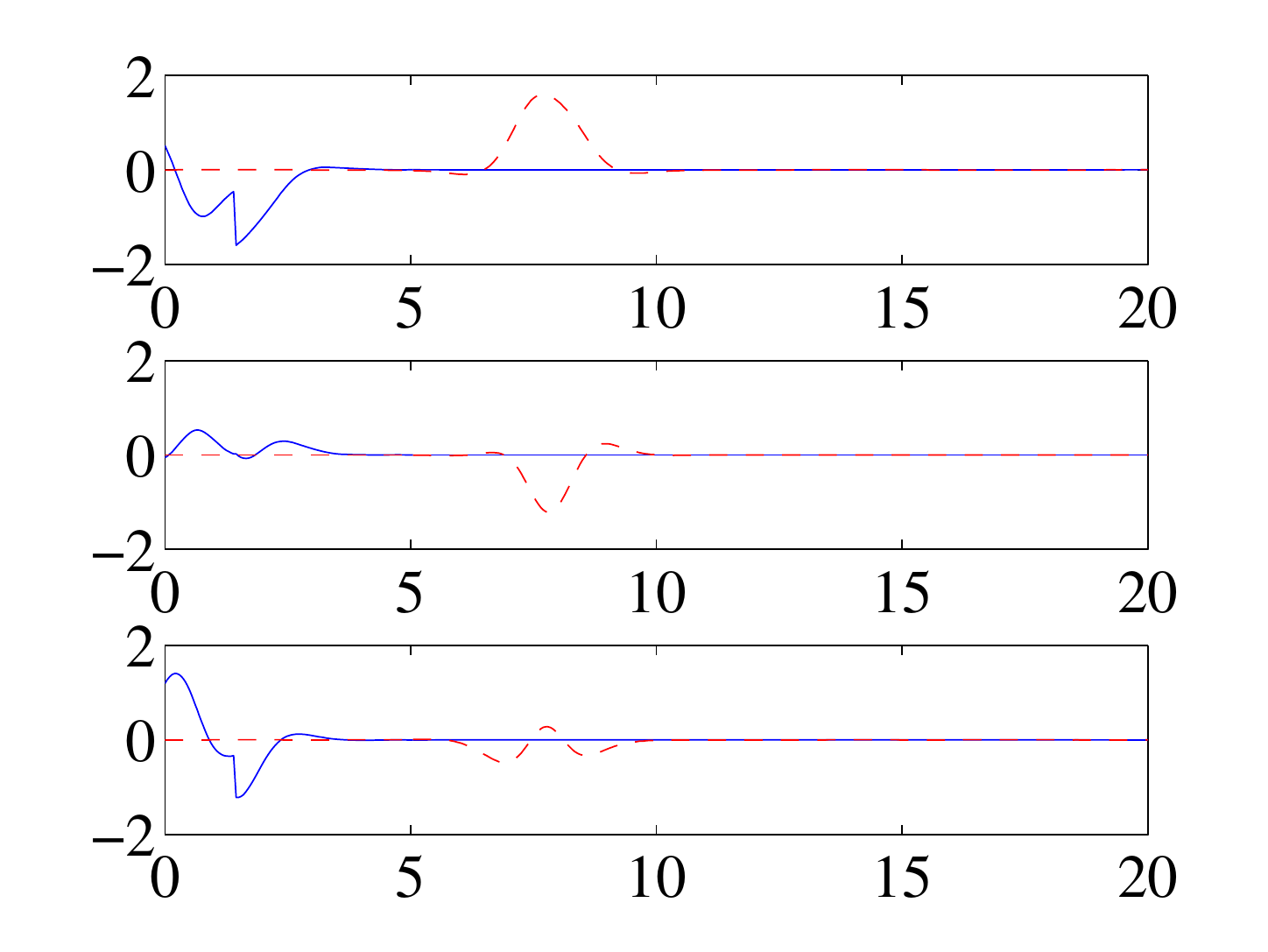}}
}
\caption{Attitude estimation without gyro bias (blue: hybrid observer, red: smooth complementary observer).}\label{fig:ob-2}
\end{figure}

\begin{figure}[h]
\centerline{
	\subfigure[Attitude estimation error $\|\bar{R}-R\|$]{\includegraphics[width=0.53\columnwidth]{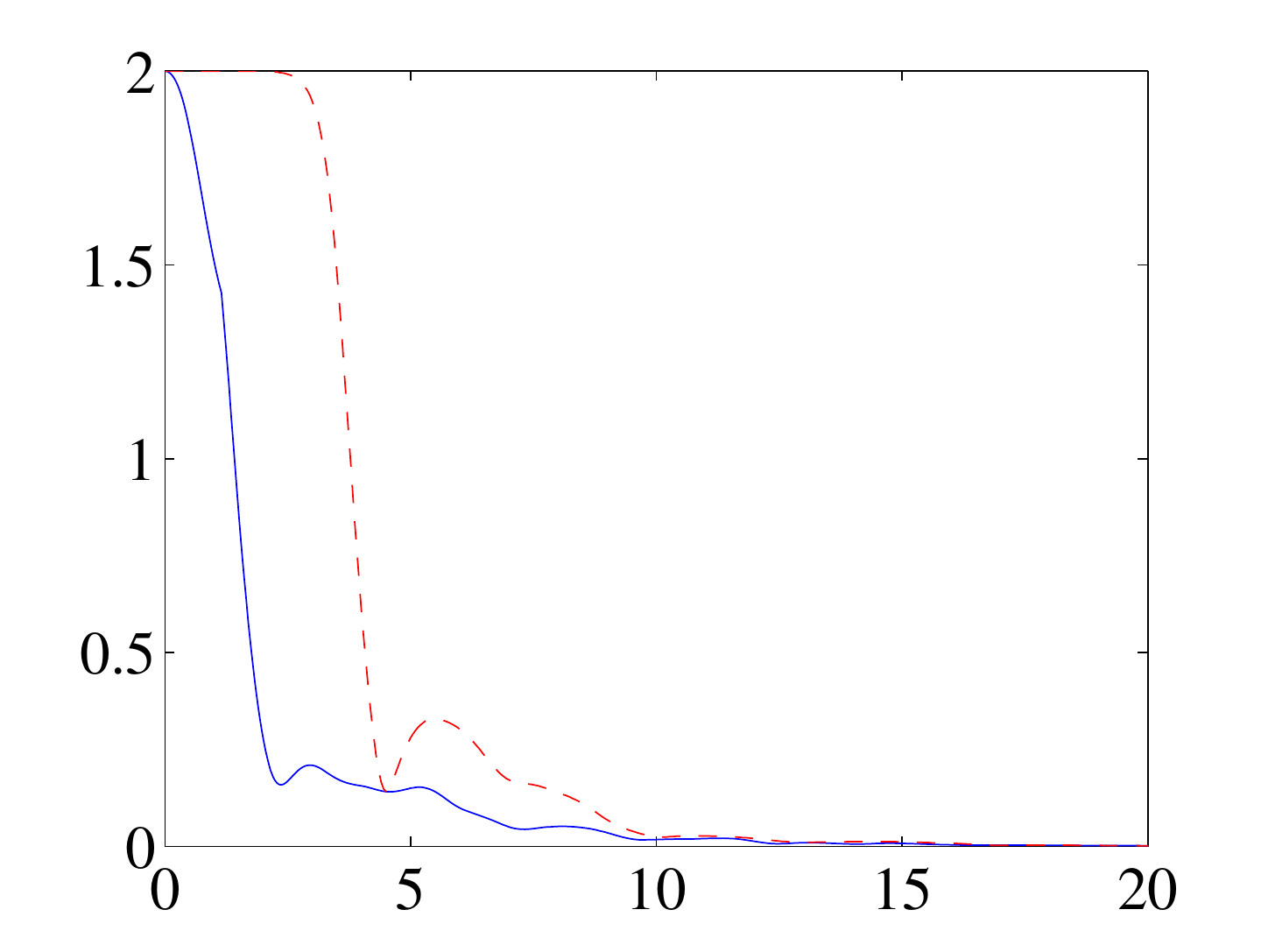}}
	\hspace*{-0.015\columnwidth}
	\subfigure[Mode change]{\includegraphics[width=0.53\columnwidth]{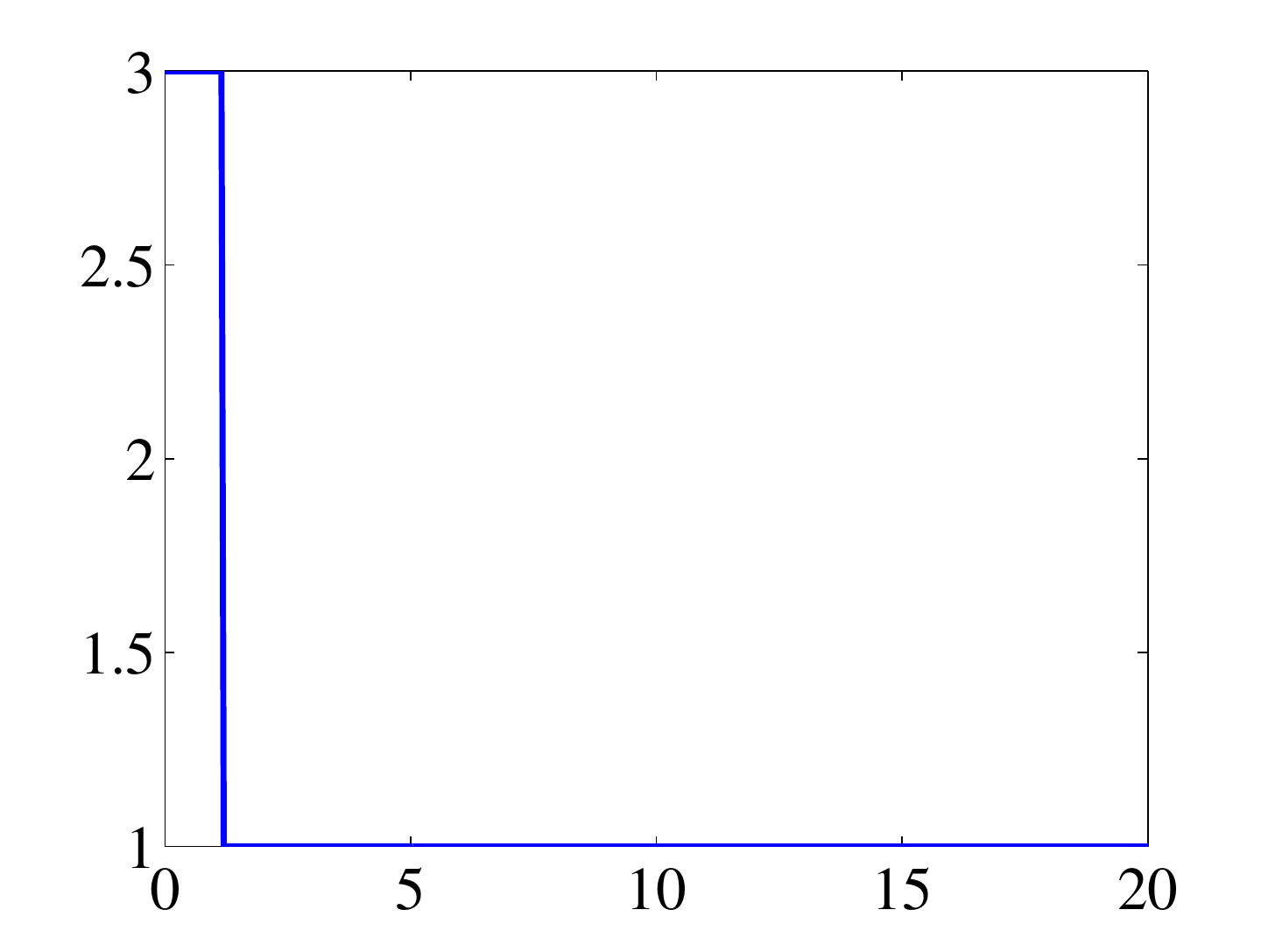}}
}	
\centerline{
	\subfigure[Innovation term $e_H(e_R)$]{\includegraphics[width=0.53\columnwidth]{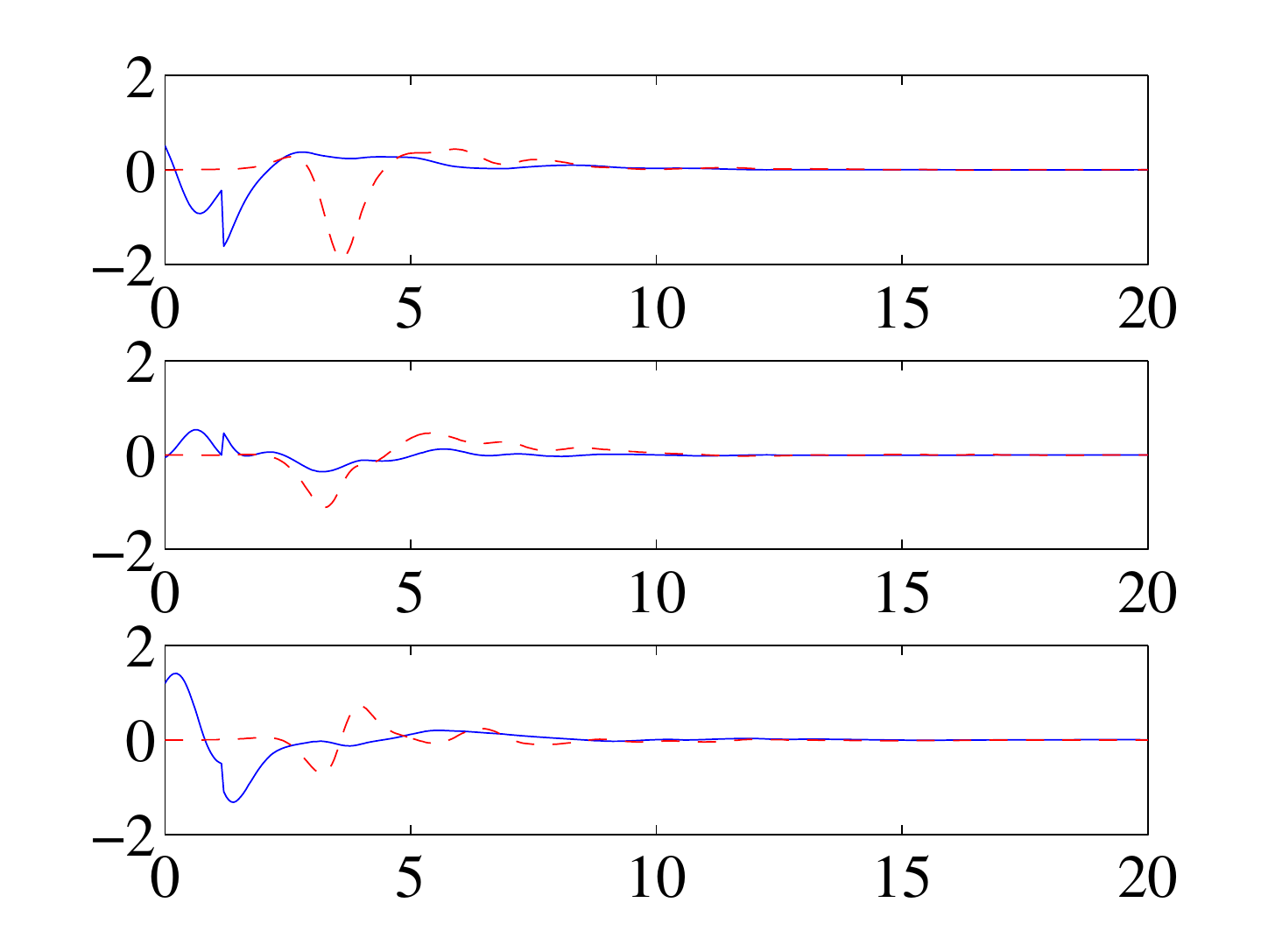}}
	\hspace*{-0.015\columnwidth}
	\subfigure[Gyro bias error $\td\ga$]{\includegraphics[width=0.53\columnwidth]{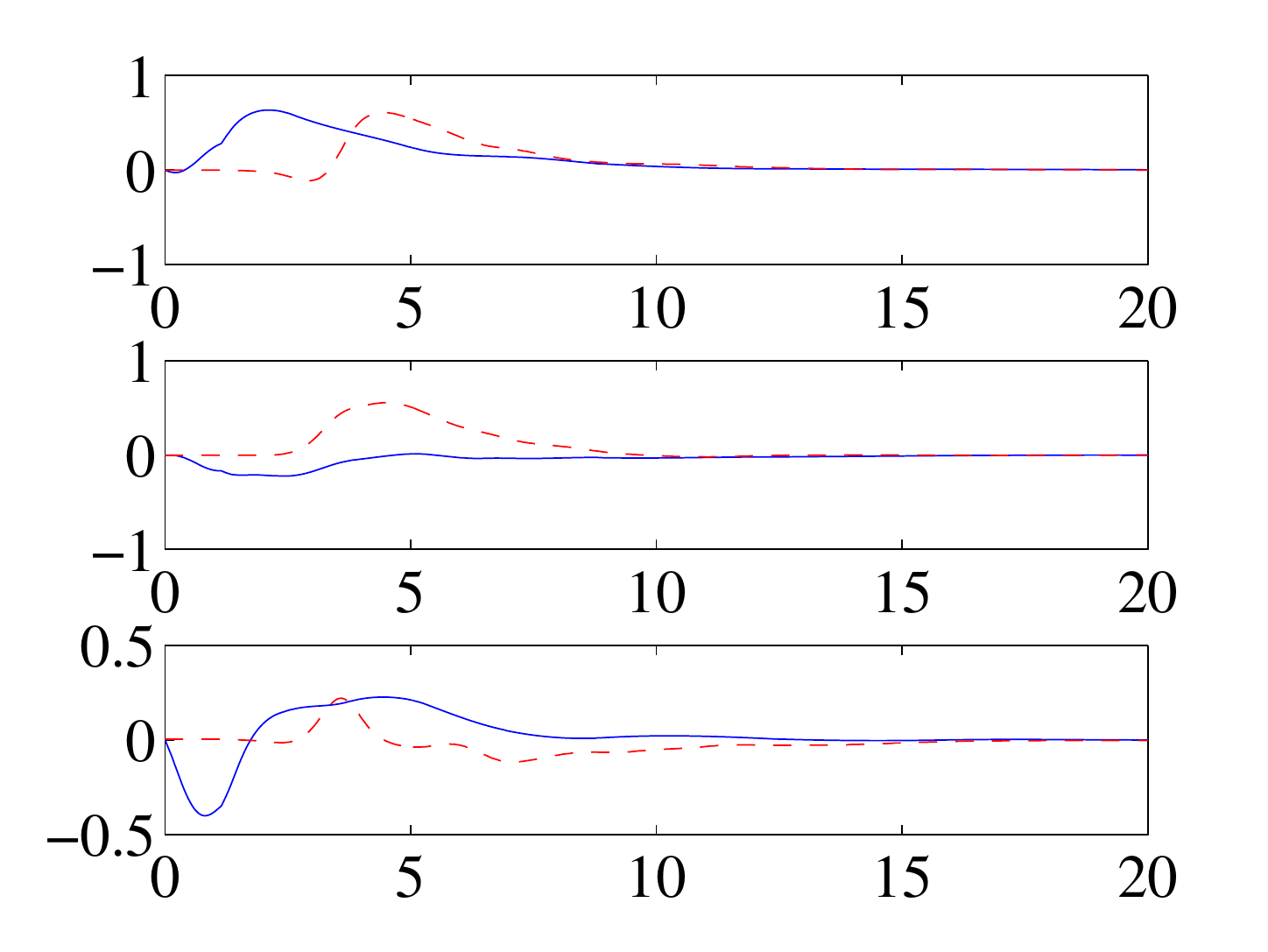}}
}
\caption{Attitude estimation with gyro bias (blue: hybrid observer, red: smooth complementary observer).}\label{fig:ob-1}
\end{figure}

\begin{figure}[h]
\centerline{
	\subfigure[60s]{\includegraphics[width=0.53\columnwidth]{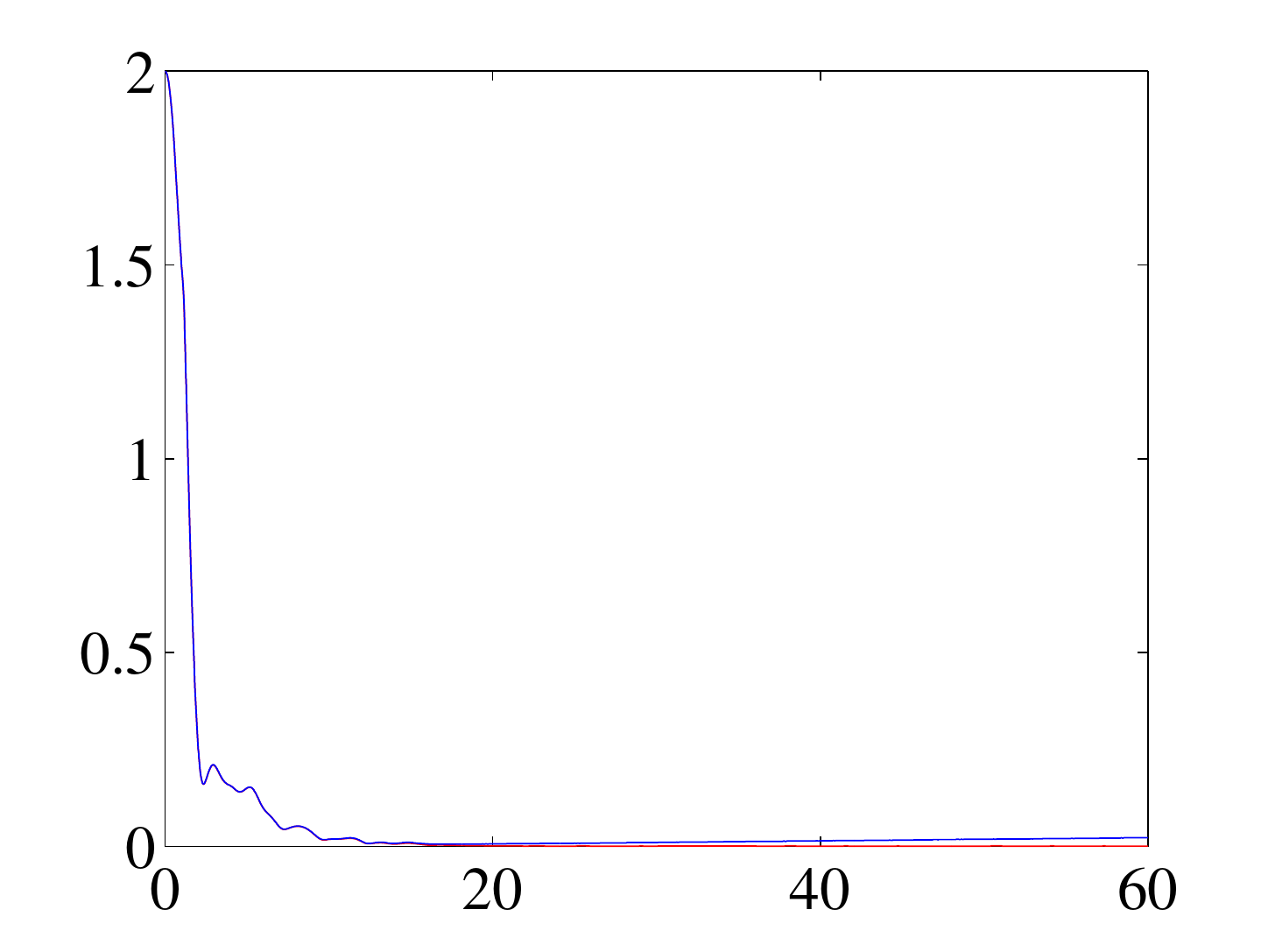}}
	\hspace*{-0.015\columnwidth}
	\subfigure[150s]{\includegraphics[width=0.53\columnwidth]{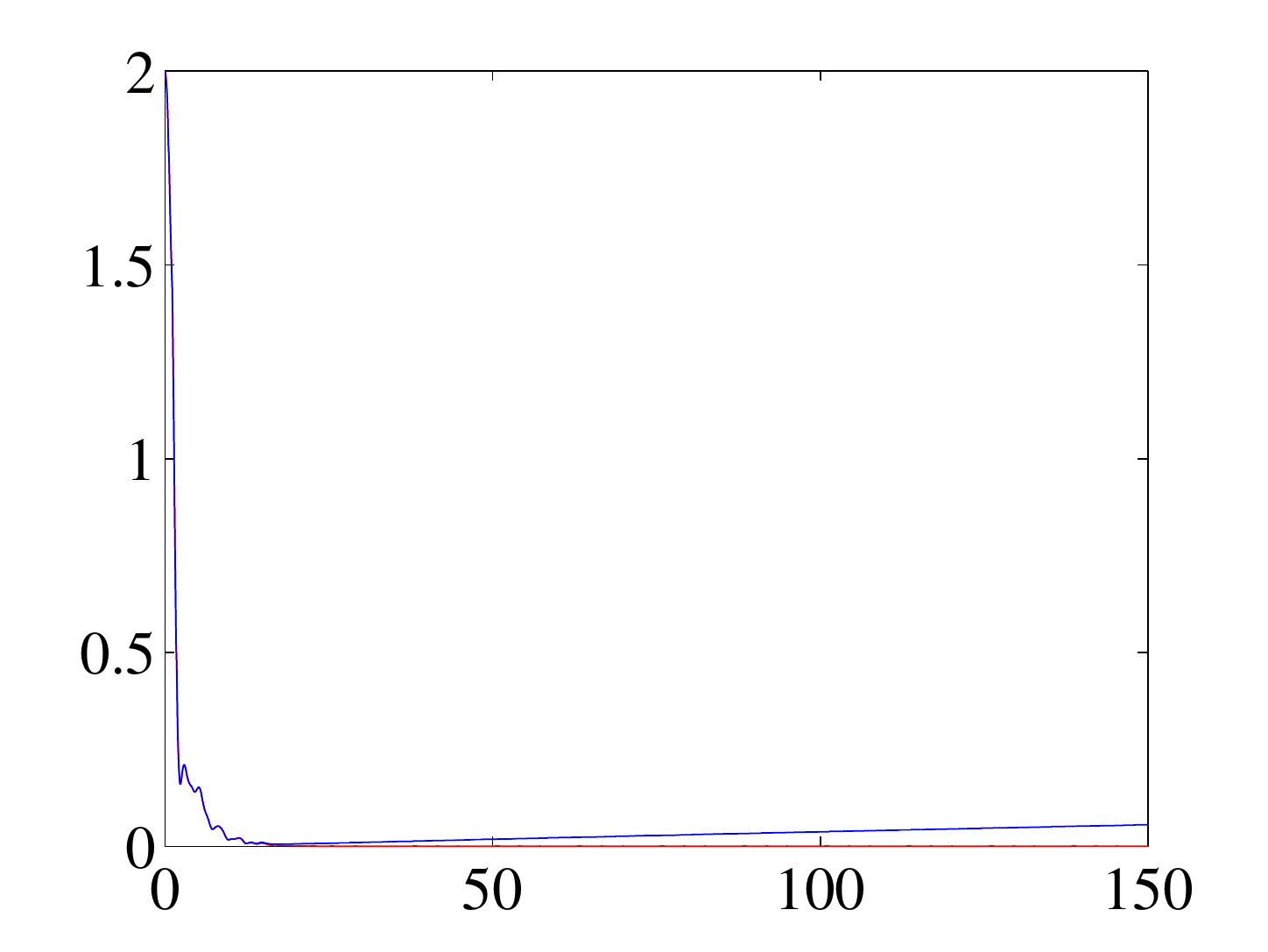}}
}	
\caption{Attitude estimation error $\|\bar{R}-R\|$ (blue: MATLAB ode45, red: geometric Runge-Kutta method).}\label{fig:compare}
\end{figure}

\section{Experimental Verification}

We also demonstrate the desirable properties of the proposed hybrid attitude observer via a fully-actuated UAV platform described in~\cite{KauCalLeeLee14}, which is attached to a spherical joint such that three attitude degrees of freedom are controlled. The control system serves to stabilize the UAV such that the mass center of the UAV is directly above the spherical joint for a fixed desired attitude $R_d=I_{3\times 3}\in\SO$, which is in fact unstable without a controller (see Figure ~\ref{fig:hexrotor}). Therefore, accurate attitude estimation is crucial to stabilizing the controlled system at this equilibrium.

\begin{figure}[h] 
\centerline{
	\hspace*{0.015\columnwidth}
	\subfigure[Initial attitude]{\includegraphics[width=0.48\columnwidth]{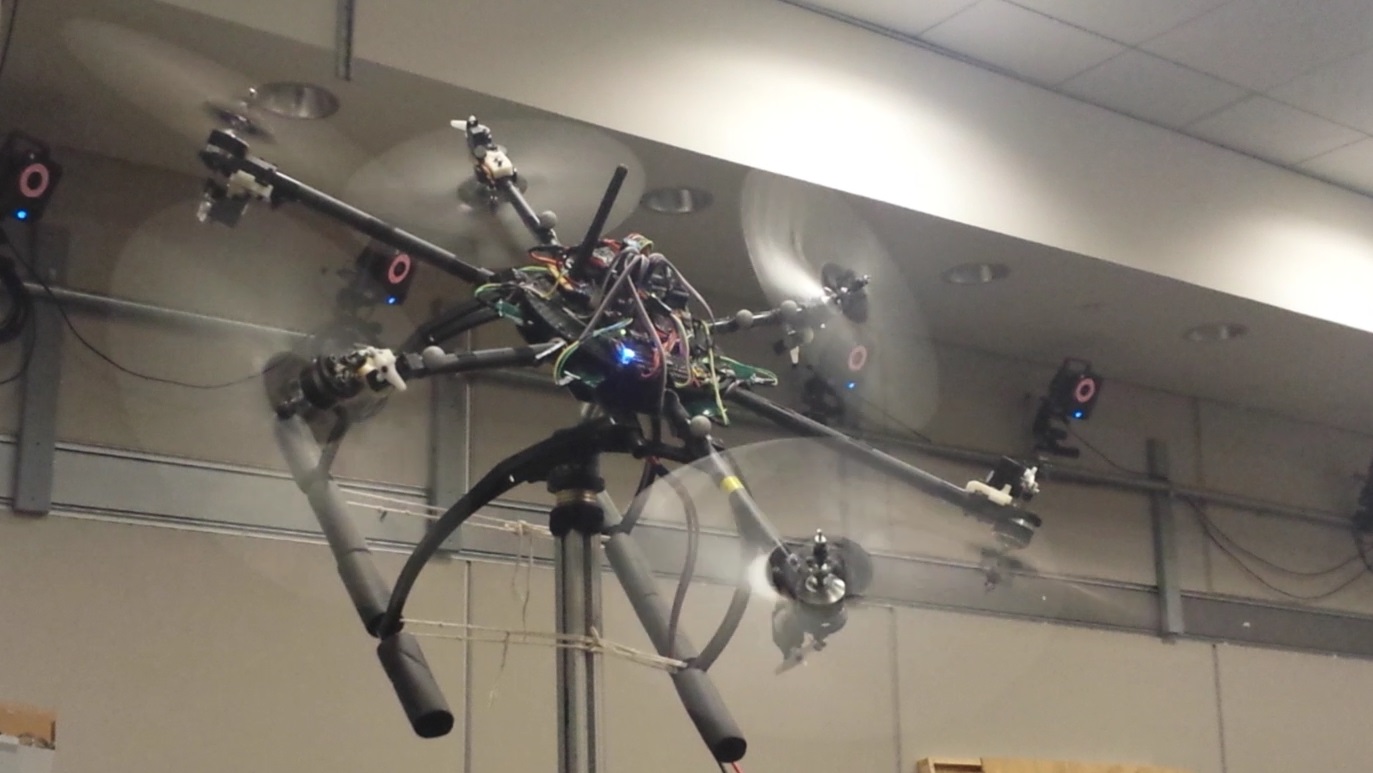}}
	\subfigure[Stabilized attitude]{\includegraphics[width=0.48\columnwidth]{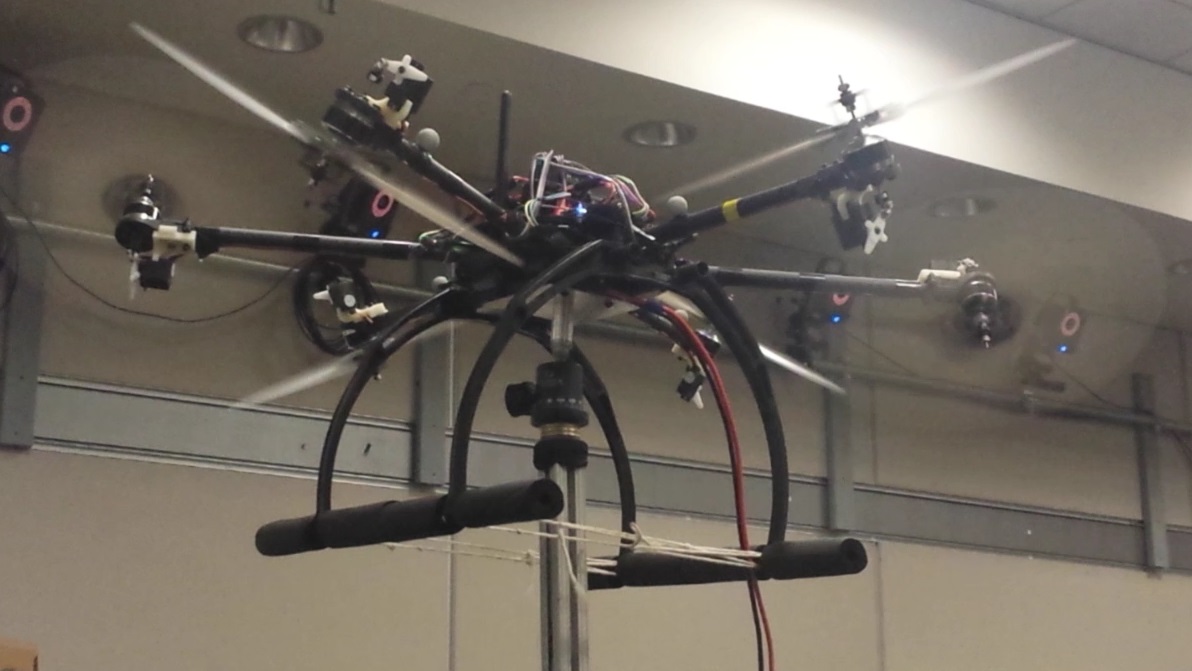}}
}	
\caption{The observer must provide accurate attitude estimates to a control system designed to stabilize the UAV above a spherical joint.}\label{fig:hexrotor}
\end{figure}


We apply both the smooth complementary attitude observer~\cite{Mahony08} and the proposed hybrid attitude observer to control the attitude, where the estimated attitude and the angular velocity measured by an IMU are used to compute the control input. 


\begin{figure}[h]
\centerline{
	\hspace*{0.015\columnwidth}
	\subfigure[Attitude estimation error $\|\bar{R}-R\|$]{\includegraphics[width=0.53\columnwidth]{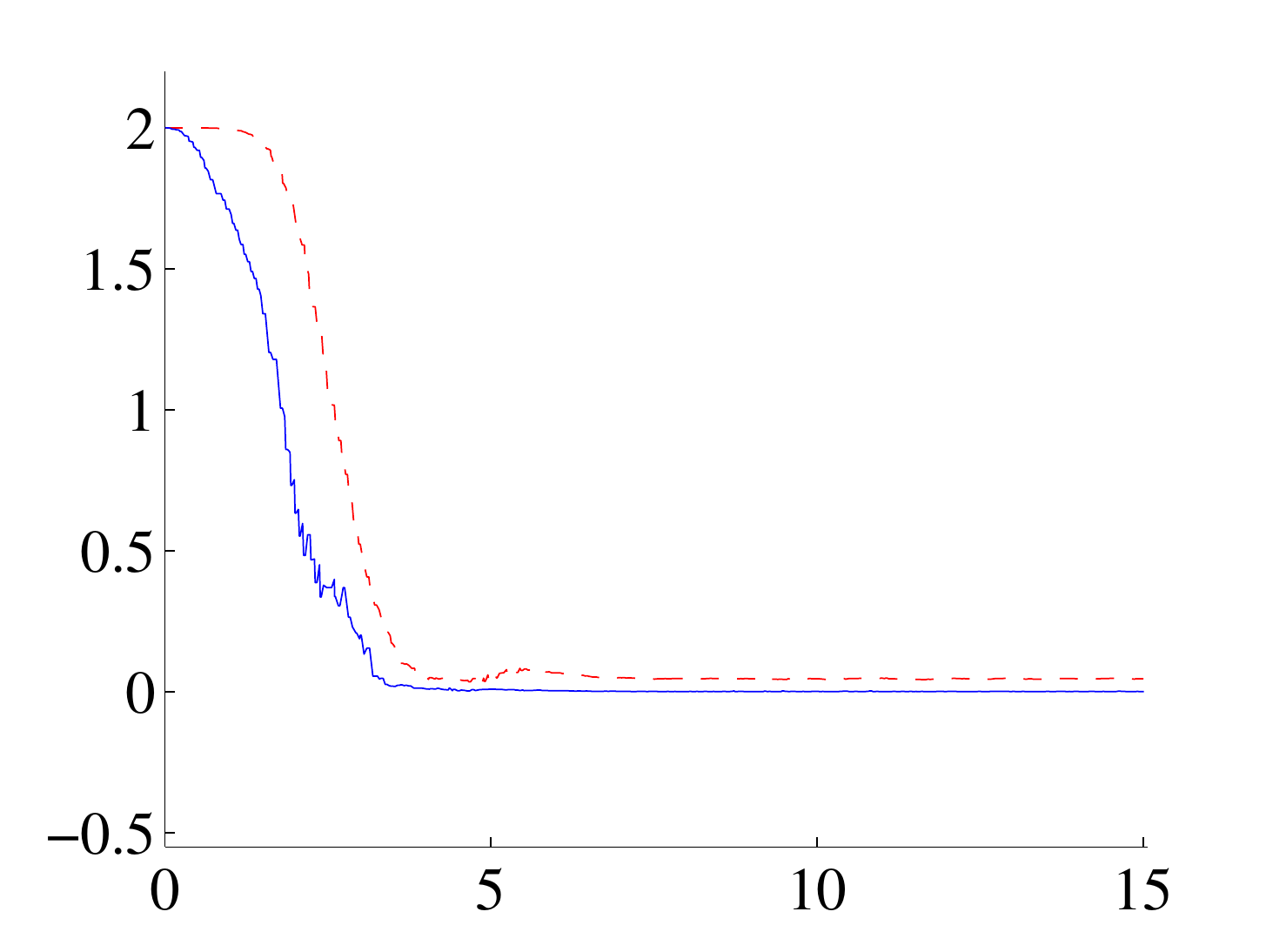}\label{fig:expdataAttErr}}
	\hspace*{-0.015\columnwidth}	
	\subfigure[Mode change]{\includegraphics[width=0.53\columnwidth]{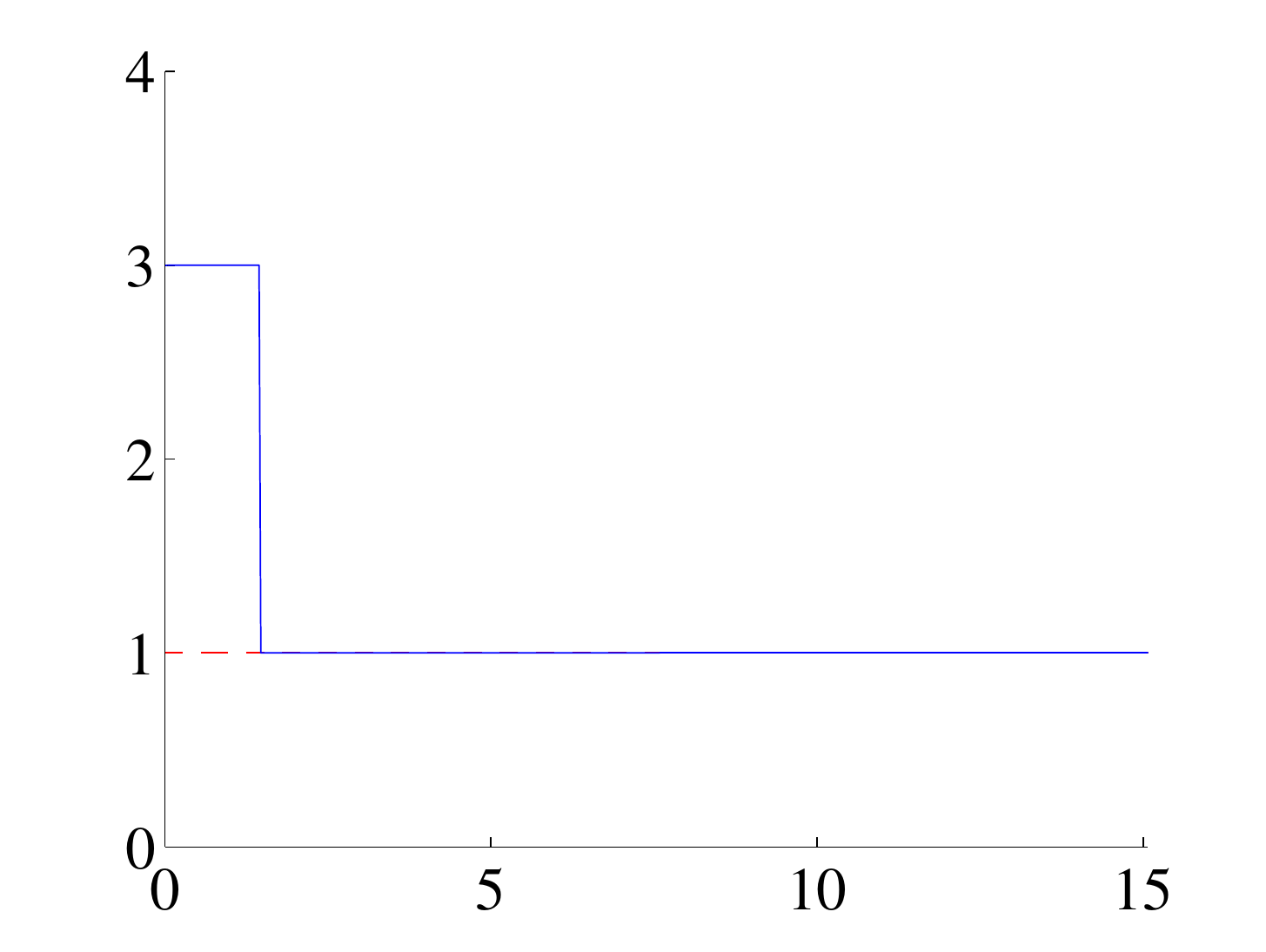}\label{fig:expdatamode}}
}	
\centerline{
	\subfigure[Hybrid innovation term $e_H$]{
		\includegraphics[width=0.53\columnwidth]{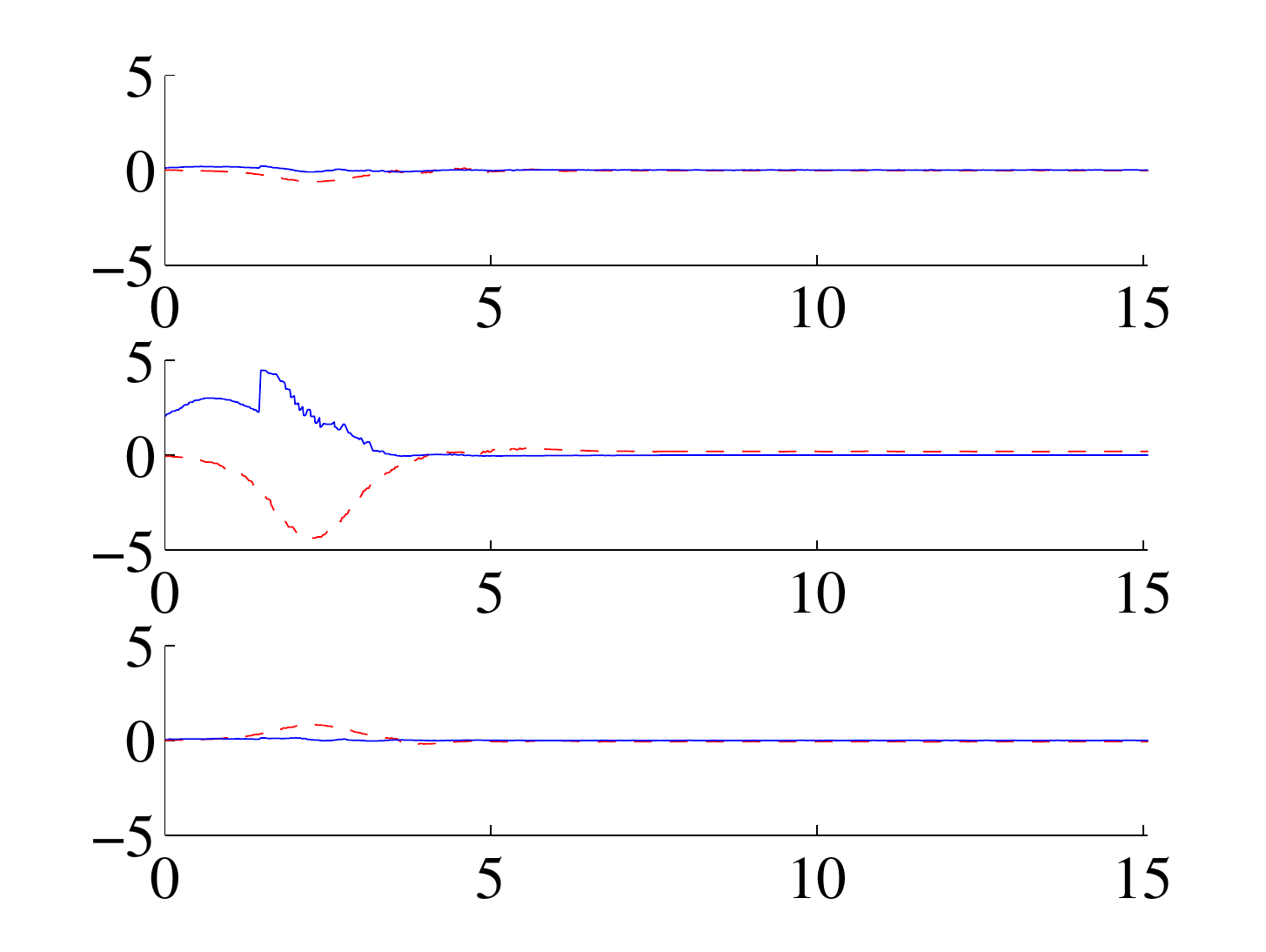}\label{fig:expdataeH}}
	\hspace*{0.015\columnwidth}
	\subfigure[Attitude control error  $\|R-R_d\|$]{
		\includegraphics[width=0.53\columnwidth]{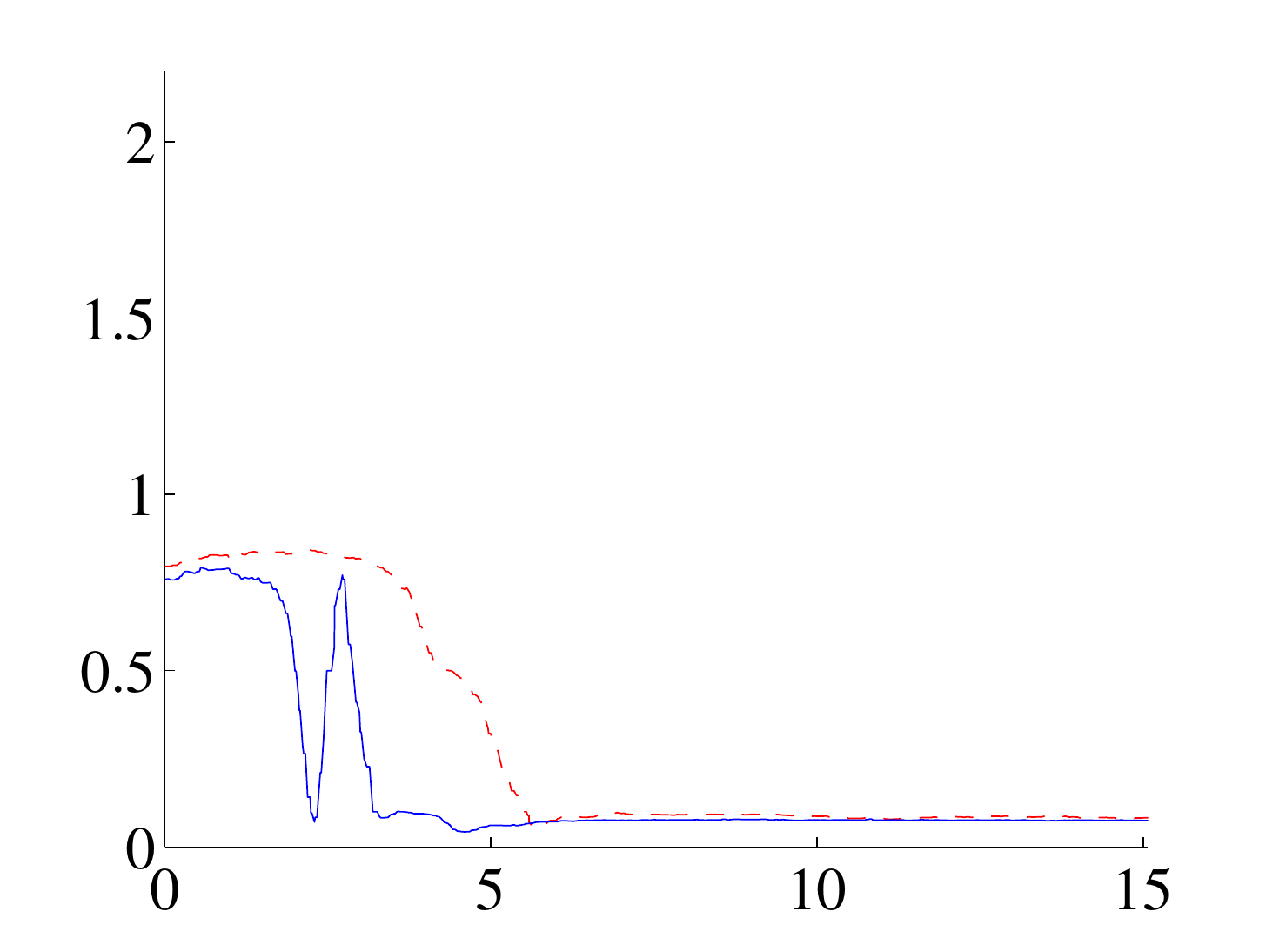}\label{fig:expdataCE}}
}
\caption{Experimental results (blue: hybrid observer, red: smooth complementary observer) show that the hybrid observer converges faster than the complementary observer.}\label{fig:expdata}
\end{figure}

The experimental results are illustrated at Figure~\ref{fig:expdata}, which show several important aspects of applying the observers to realistic scenarios. 
It is illustrated that the proposed hybrid attitude observer exhibits considerably faster initial convergence rate for the attitude estimation error. Initially, the hybrid observer is at the mode $\Rc$, causing a steeper decrease in attitude error than the complementary observer, which is confined to the mode $\Ra$ at Figure~\ref{fig:expdataAttErr}.

The experimental results also demonstrate how the performance of the control system depends heavily on the estimate provided by the observer at Figure \ref{fig:expdataCE}. Indeed, when the UAV depends on the smooth complementary observer, the controlled system requires roughly $50\%$ more time to stabilize than when the hybrid observer is used because the smooth complementary observer provides a more inaccurate attitude estimate for a longer time period.

\section{Conclusions}
It is well known that attitude estimators suffer from the topological restriction inherent to the special orthogonal group that prohibits any smooth estimator to achieve global attractivity. We demonstrate that this may cause significant performance degradation at certain cases, and motivated by this, we propose a hybrid attitude observer that guarantees global asymptotic stability. This exhibits desirable properties that the convergence rate is substantially faster, and the estimated attitudes evolve on the special orthogonal group. These are illustrated by numerical examples and an experiment.


\appendix[]	
\subsection{Proof for Proposition 2}

\paragraph{Equilibrium Configuration}

For each configuration error function $\Psi_\m$, there are four equilibrium configurations that corresponds to the critical points of $\Psi_\m$. Hence, there are twelve critical points, including the single desired equilibrium and eleven undesired critical points. The values of $(\bar b_1,\bar b_2,\bar b_3)$ for each critical point are summarized at Table I. 

We first show that those undesired critical points cannot become an equilibrium of the proposed hybrid attitude observer, as they belong to the jump set. For example, at the third undesired critical point of the nominal mode $\m=\Ra$, when $(\bar{b}_1,\bar{b}_2,\bar{b}_3)=(-b_1,-b_2,b_3)$, we have
	\begin{align*} 
	\Psi_\Ra=2(\lam_1+\lam_2),~ \Psi_\Rb=2\lam_1+\lam_2\al,~\Psi_\Rc=\lam_1\al+2\lam_2,
	\end{align*}
where $\Psi_\Ra>\Psi_\Rb>\Psi_\Rc$ as $\al<2$ and $\lam_1>\lam_2>\lam_3$. Therefore, $\rho=\Psi_\Rc$ and $\Psi_\Ra-\rho=\lam_1(2-\al)\geq\de$ from \refeqn{de}. Hence, the corresponding undesired critical point belongs to the  jump set $\D$ defined at \refeqn{DD}, and it cannot become an equilibrium. This can be repeated to show that all of the eleven undesired equilibria belong to the jump set as well, and the only equilibrium of the proposed hybrid attitude observer is the desired equilibrium.

	\begin{table}\label{table:comp1} 
	\caption{Configurations of critical points}
	\begin{center} 
	\begin{tabular} {|c|c|c|c|}
	\hline 
	$\Psi_\m$ & Critical Point & $(\bar{b}_1,\bar{b}_2,\bar{b}_3)$  & $\Psi_\m-\rho$ \\	
		\hline 
	\multirow{4}{*}{$\Psi_\Ra$} & Desired & $(b_1,\,b_2,\,b_3)$ & $0$\\
	 & Undesired 1 & $(-b_1,\,b_2,\,-b_3)$  & $\lam_1(2-\al)$\\	
	 & Undesired 2 &	$(b_1,\,-b_2,\,-b_3)$ & $\lam_2(2-\al)$\\
	 & Undesired 3 &	$(-b_1,\,-b_2,\,b_3)$ & $\lam_1(2-\al)$\\
	\hline 
	\multirow{4}{*}{$\Psi_\Rb$} & Undesired 1& $(b_1,\,-b_3,\,b_2)$ &  $\lam_2\mu$\\ 
	 & Undesired 2 & $(-b_1,\,-b_3,\,-b_2)$ &   $\lam_1(2-\al)+\lam_2\mu$\\	
	 & Undesired 3 & $(b_1,\, b_3,\,-b_2)$	&   $\lam_2\mu$\\
	 & Undesired 4 & $(-b_1,\,b_3,\,b_2)$	&  $\lam_1(2-\al)+\lam_2\nu$\\	
	\hline		
	\multirow{4}{*}{$\Psi_\Rc$} & Undesired 1 & $(-b_3,\,b_2,\,b_1)$ &   $\lam_1(\al-\be-1)$\\ 
	 & Undesired 2 & $(-b_3,\, -b_2,\,-b_1)$ &   $\lam_1\mu+\lam_2(2-\al)$\\	
	 & Undesired 3 & $(b_3,\, b_2,\,-b_1)$ &  $\lam_1\mu$\\
	 & Undesired 4 & $(b_3,\, -b_2,\,b_1)$	& $\lam_1\nu+k_2(2-\al)$\\	
	\hline		
	\end{tabular}
	\end{center} 
	\end{table} 


\paragraph{Time-derivative of $\Psi_{\m}$}

Next, we derive the time-derivative of $\Psi_{\m}$ for each mode in the flow set. From \refeqn{sRb},
 	\begin{align*} 
 	\dot{b}_i &=\dot{R}^\T s_i =-\hat\W R^\T s_i =-\hat{\W}b_i, \\
 	\dot{\bar b}_i &=\dot{\bar R}^\T s_i =-[(\W_y-\bar{\ga})+k_Re_{H_i}]^\wedge\bar{b}_i,
 	\end{align*}
From \refeqn{Ni}-\refeqn{E2}, the time-derivative of $\Psi_{N_i}$ is 
	\begin{align*} 
	\dot{\Psi}_{N_i}
	&= -(\dot{\bar b}_i)^\T b_i -\bar{b}_i^\T\dot{b}_i \\
	&= -{\bar b}_i^\T(\W+\td{\ga}+k_Re_{H_i})^\wedge b_i +{\bar b}_i^\T\hat\W b_i\\
	&= -{\bar b}_i^\T\hat{\W}b_i -{\bar b}_i^\T\hat{\td{\ga}}b_i -k_R\bar{b}_i^\T(\hat{e}_{H_i}b_i) +{\bar b}_i^\T\hat\W b_i\\
	&= \td{\ga}^\T(\bar{b}_i\t b_i)+k_Re_{H_i}^\T(\bar{b}_i\t b_i),
	\end{align*}
and the time-derivatives of $\Psi_{E_1}, \Psi_{E_2}$ are given by 
	\begin{align*} 
	\dot{\Psi}_{E_i}
	&= \be[(\dot{\bar b}_i)^\T b_3 +\bar{b}_i^\T\dot{b}_3] \\
	&= \be[{\bar b}_i^\T(\W+\td{\ga}+k_Re_{H_i})^\wedge b_3 -{\bar b}_i^\T\hat\W b_3]\\
	&= \be[{\bar b}_i^\T\hat{\W}b_3 +{\bar b}_i^\T\hat{\td{\ga}}b_3 +k_R\bar{b}_i^\T(\hat{e}_{H_i}b_3) -{\bar b}_i^\T\hat\W b_3]\\
	&= \td{\ga}^\T \be(b_3\t\bar{b}_i)+k_Re_{H_i}^\T\be(b_3\t \bar{b}_i),
	\end{align*}	
for $i=1,2$. Therefore, combining these, 
	\begin{align} 
	\dot\Psi_\Ra
	&=\sum_{i=1}^3 \lam_i\dot\Psi_{N_i} =(\td{\ga}+k_Re_{H})^\T\left(\sum_{i=1}^3\lam_i\bar{b}_i\t b_i\right)\no\\ &=-(\td{\ga}+k_Re_{H})^\T e_{H} =-\td{\ga}^\T e_{H}-k_R\|e_{H}\|^2,\label{eq:dpsi}
	\end{align}
for $\m=\Ra$. Also, for $\m=\Rb$, we can show that
	\begin{align*} 
	\dot{\Psi}_\Rb 
	&= \lam_1\dot{\Psi}_{N_1} +\lam_2\dot{\Psi}_{E_2} +\lam_3\dot{\Psi}_{N_3} \\
	&= (\td{\ga}+k_Re_H)^\T[\bar{b}_1\t b_1 +\be(b_3\t\bar{b}_2) +\bar{b}_3\t b_3]\\
	&=-\td{\ga}^\T e_{H}-k_R\|e_{H}\|^2,
	\end{align*}
which can be repeated for $\dot{\Psi}_\Rc$ to conclude that 
	\begin{align} 
	\dot{\Psi}_\m=-\td{\ga}^\T e_{H}-k_R\|e_{H}\|^2,\label{eq:Pmdot}
	\end{align}
for any $\m=\{\Ra,\Rb,\Rc\}$.

\paragraph{Stability Proof}

Define a Lyapunov function as
	\begin{align*} 
	\V_\m=\Psi_\m +\frac{1}{2k_I}\|\td{\ga}\|^2,
	\end{align*}
which is positive-definite about the desired equilibrium $\bar R =R$ with $\tilde\gamma=0$ at the mode $\m=\Ra$.

We first analyze the change of the Lyapunov function when restricted to the flow set as follows. From \refeqn{gob2} and \refeqn{Pmdot}, the time-derivative of $\V$ is given by
	\begin{align} 
	\dot{\V}
	&= \dot{\Psi}+\frac{1}{k_I}\td\ga^\T\dot{\td \ga} =-\td{\ga}^\T e_H-k_R\|e_H\|^2 +\td\ga^\T e_R \no\\
	&=-k_R\|e_H\|^2, \label{eq:dV}
	\end{align}
which implies that the Lyapunov function is non-increasing, $\lim_{t\rightarrow\infty}\|e_H\|=0$, and $\|\tilde\gamma\|$ is uniformly bounded. Further, with the assumption that $\|\W\|$ and $\|\dot\W\|$ are bounded, one can show that $\|\ddot{e}_H\|$ is uniformly bounded as well, which follows $\lim_{t\rightarrow\infty} \dot e_H =0$, from Barbalat's Lemma~\cite[Lemma 8.2]{Kha96}.

For $\mb=\Ra$, these can be used to shown the convergence of the gyro bias error as follows. From \refeqn{widehat} and \refeqn{ob3}, the innovation term $e_H$ can be written as 
	\begin{align*} 
	\hat{e}_H
	&=\sum_{i=1}^nk_i(v_i^B\t v_i^E)^\wedge =\sum_{i=1}^nk_iv_i^E{v_i^B}^\T-v_i^B{v_i^E}^\T \\
	&= \bar{R}^\T\left(\sum_{i=1}^nk_i v_i^I{v_i^I}^\T \right)R -{R}^\T\left(\sum_{i=1}^nk_i v_i^I{v_i^I}^\T \right)\bar{R} \\
	&=\bar{R}^\T KR-R^\T K\bar{R}.
	\end{align*}
From \refeqn{eom2}, \refeqn{gob1}, the time-derivative of $\hat{e}_H$ is 
	\begin{align*} 
	\dot{\hat{e}}_H
	&= -(\W+\td{\ga}+k_Re_H)^\wedge\bar{R}^\T KR +\bar{R}^\T KR\hat{\W} +\hat{\W}R^\T K\bar{R} \\
	&\quad -R^\T K\bar{R}(\W+\td{\ga}+k_Re_H)^\wedge.
	\end{align*} 
Using \refeqn{hat3}, it is rewritten as
	\begin{align} 
	\dot{e}_H
	&= -(\tr[\bar{R}^\T KR]\I-\bar{R}^\T KR)(\td{\ga}+k_Re_H) +\hat{e}_H\W. \no 
	\end{align}
Since $\lim_{t\rightarrow\infty}\|\dot{e}_H\|=\lim_{t\rightarrow\infty}\|e_H\|=0$, 
	\begin{align} 
	\lim_{t\rightarrow\infty}\|(\tr[\bar{R}^\T KR]\I-\bar{R}^\T KR)\td\ga\|=0. \label{eq:lim}
	\end{align}
From \refeqn{K}, the matrix in the above equation is written as
	\begin{align*} 
	&\tr[\bar{R}^\T KR]\I-\bar{R}^\T KR \\
	&\quad=\bar{R}^\T U(\tr[G\td{U}]\I-G\td{U})U^\T\bar{R},
	\end{align*}	
where $\td{U}=U^\T{R}\bar{R}^\T U\in\SO$. One can show $\mathrm{rank}(\tr[G\td{U}]\I-G\td{U})=3$ when $\tilde U=I$, such that $\tr[\bar{R}^\T KR]\I-\bar{R}^\T KR$ also has the full rank in the limit as $t\rightarrow \infty$. Then, \refeqn{lim} implies  $\lim_{t\rightarrow\infty}\|\td\ga\|=0$. In short, the Lyapunov function asymptotically converges to zero for $\mb=\Ra$. Together with \refeqn{dV}, these can be repeated on other two modes, to show that the Lyapunov function asymptotically converges at any mode of the flow set. 

Next, the change of the Lyapunov function over the jump from a mode $\m$ to a new mode $\mathcal{G}(\m)$ given at \refeqn{GG} is
	\begin{align*} 
	\V_\G-\V_\m =\rho-\Psi_\m \leq -\de,
	\end{align*}	
where the last inequality is from the definition of the jump set \refeqn{DD}. In other words, the Lyapunov function strictly decreases over any jump.

From these, it follows that the desired equilibrium $(\bar R, \tilde\gamma)=(R,0)$ is globally asymptotically stable, and the number of jumps is finite.

\bibliography{CDC15}
\bibliographystyle{IEEEtran}
\end{document}